\documentclass{amsart}
\usepackage{amsfonts,amssymb,amsmath,amsthm}
\usepackage{url}
\usepackage{enumerate}

\newcommand{\assign}{:=}
\newcommand{\dueto}[1]{\textup{\textbf{(#1) }}}
\newcommand{\nin}{\not\in}
\newcommand{\nocomma}{}
\newcommand{\noplus}{}
\newcommand{\nosymbol}{}

\newenvironment{itemizedot}{\begin{itemize} 
}{\end{itemize}}

\newtheorem{thrm}{Theorem}[section]
\newtheorem{lem}[thrm]{Lemma}
\newtheorem{prop}[thrm]{Proposition}

\theoremstyle{definition}
\newtheorem{definition}[thrm]{Definition}

\numberwithin{equation}{section}
\title{On the stationary solutions of Doi-Onsager model in general dimension}
\author{Mohammad Niksirat}
\address{Department of Mathematics, University of Toronto, Toronto, Canada, M5S 2E4}
\email{niksirat@math.toronto.edu}
\begin{document}

\begin{abstract}
  We give new results of the phase transition of dilute colloidal solutions of
  rod-like molecules in dimension $D \geq 3$. For the low concentration of
  particles in a carrier fluid, we prove that the isotropic phase is the unique solution to the
  Doi-Onsager model with the general potential kernel. In addition, we present the regime of the bifurcation 
  of nematic phases in the class of axially symmetric solutions. Our method is based on a generalization of 
  the classical Leray-Schauder degree we developed for this problem.
\end{abstract}

\maketitle
\section{Introduction}

In 1949, L. Onsager \cite{Onsager1949} proposed a mathematical model for
the phase transition of dilute colloidal solutions of rod-like molecules. As
the fluid in both phases is homogeneous, Onsager's theory focuses on a
probability density function $f (r)$ over the unit sphere $S^2 \in
\mathbb{R}^3$. Let $f (r) : S^2 \rightarrow [0, \infty)$ be the probability
density function of the directions of the rod-like molecules, that is, for any
$A \subset S^2$
\begin{equation}
  P ( \text{\rm {the rod is along}} \hspace{0.25em} r
  \in A) = \int_A f (r) d \sigma .
\end{equation}
As we are modeling rod-like molecules with no distinction between the two
ends, we can further assume $f (r) = f (- r)$. Consequently, the constraints
on $f (r)$ are
\begin{equation}
  \label{ff} f (r) \geq 0, \hspace{0.4cm} f (r) = f (- r), \hspace{0.4cm}
  \int_{S^2} f (r) d \sigma = 1.
\end{equation}
The mean interaction potential between molecules is expressed by
\begin{equation}
  \label{U2} U (f) (r) = \lambda \int_{S^2} K (r, r') f (r') d \sigma (r'),
\end{equation}
where $\lambda$ can be interpreted as either the concentration of the
particles in the carrier fluid or equivalently the inverse of the absolute
temperature, and the potential kernel $K$ is defined by Onsager as
\begin{equation}
  \label{Onsager-K} K (r, r') = |r \times r' | .
\end{equation}
With this interaction field, Onsager suggested that the possible phases of a
liquid crystal are the critical points of the following energy functional
$\mathcal{E}$:
\begin{equation}
  \label{free-energy} \mathcal{E} (f) = \int_{S^2} f (r)  \left( \log \nocomma
  f (r) + \frac{1}{2} U (f) (r) \right) d \sigma .
\end{equation}
By the classical variational method, it is simply seen that the density
function $f$ is a minimizer of (\ref{free-energy}) if the functional
\begin{equation}
  V (f) \assign \log f + U (f),
\end{equation}
is constant. By (\ref{ff}), this is in turn equivalent to the equation
\begin{equation}
  \label{f2} f (r) = \left( \int_{S^2} e^{- U (f)} d \sigma \right)^{- 1} e^{-
  U (f) (r)} .
\end{equation}
On the other hand, according to the relation
\begin{equation}
  \Delta_r f + {\rm div} (f \nabla_r U (f)) = {\rm div} (f \nabla_r V),
\end{equation}
it is simply seen that the solutions of (\ref{f2}) are the steady state
solutions of the Doi equation:
\begin{equation}
  \label{dyn} \frac{\partial f}{\partial t} = \Delta_r f +  {\rm div}
  (f \nabla_r U (f)) .
\end{equation}
Apparently, (\ref{f2}) has trivial solutions $\bar{f} = \frac{1}{vol (S^2)}$
that are correspond to the uniform distribution of molecules without any
preferred directional order. This is called an isotropic phase of the fluid.
Approximating $K$ by some lower orders terms, Onsager was able to show a
transition to a non-uniform state called nematic phases in the case when
$\lambda$ passes a critical value.

More quantitative analysis of the system $\left( \ref{U2} \right), \left(
\ref{f2} \right)$ with the Onsager kernel turned out to be difficult. On the
other hand there are kernels capturing the qualitative behavior of the
solution that are more friendly to mathematical analysis. One such kernel, due
to Maier and Saupe, reads
\begin{equation}
  \label{K-MS} K (r, r') = |r.r' |^2 - {\rm constant},
\end{equation}
that is usually written as $K (r, r') = \cos^2 \gamma$ if the constant is
discarded. The advantage in considering (\ref{K-MS}) instead of the Onsager
kernel (\ref{Onsager-K}) is that the Maier-Saupe kernel is the eigenvector of
the Laplace-Beltrami operator on $S^2$ and then lies in a finite dimensional
space \cite{Chen2010}. This reduces the infinite dimensional problem
(\ref{U2})-(\ref{f2}) to a finite dimensional nonlinear system
of equations. This reduced system, still highly nontrivial, is nevertheless
more tractable than the original system. As a consequence, (\ref{U2})-(\ref{f2}) with the Maier-Saupe kernel has been well understood through
brilliant works of many researchers, see \cite{Const04,Fatkullin2005A,Liu05,Zhou05,Liu07,Zhou07} for the model in $\mathbb{R}^3$; also see
\cite{Const07,Fatkullin2005A,Luo05} for the model in
$\mathbb{R}^2$, and \cite{Wang2008} for the general dimensional case
$\mathbb{R}^D$. Inspired by these works, equations (\ref{U2}) and (\ref{f2})
with other kernels enjoying similar dimension reduction property has also be
analyzed, see e.g. \cite{Chen2010}.

Fatkullin and Slastikov \cite{Fatkullin2005A,Fatkullin2005}
completely classified the solution of the Onsager equation with the
Maier-Saupe kernel and for the anti-symmetric kernel \[ K(r, r') = - r.r'\] on
$S^1$ and $S^2$. Instead of $\lambda$, they presented their results in terms
of the temperature $\tau$, however since $\lambda$ and $\tau$ are inversely
proportional their results hold for the original case. In particular they
obtained the exact nematic solutions in $S^2$ for the problem with the
Maier-Saupe kernel as:
\begin{equation}
  f (\varphi, \theta) = \beta^{- 1} e^{- r_{1, 2} (\tau)  (3 \cos^2 \theta -
  1)},
\end{equation}
and for the anti-symmetric kernel they obtained
\begin{equation}
  f (\varphi, \theta) = \beta^{- 1} e^{- r (\tau) \cos \nocomma \theta} .
\end{equation}
In addition they presented some results of the stability of the above
solutions. Luo et al \cite{Luo2005} considered the Maier-Saupe
interaction kernel on $S^1$ and proved that for the potential strength
$\lambda \leq 4$, the isotropic solution $\bar{f} = \frac{1}{2 \pi}$ is the
unique solution of the equation. The nematic solution will bifurcates when the
the liquid crystal cool down or equivalently the potential strength increase
to $\lambda > 4$. They also proved that all nematic solutions are obtained by
an arbitrary rotation from a $\pi$ periodic nematic solution. At the same
time, Liu et al \cite{Liu2005} obtained an explicit solution to
(\ref{U2}) and (\ref{f2}) on $S^2$ with the Maier-Saupe kernel and determined the
bifurcation regime of the solution. The solution is of the following form for
a director $y$ and constant $k$ and
\begin{equation}
  f (x) = ke^{- \eta (x.y)^2} .
\end{equation}
With the Maier-Saupe model understood, interest in the original Onsager model
was resurrected. Much progress has been made in the past few years in the case
$D = 2$. In \cite{Chen2010}, the axisymmetry of all possible solutions is
proved in the sense that for any solution $f(\theta)$, there is $\theta_0$ such that
$f(\theta_0) = f (\theta_0 + \pi)$. It is also proved in \cite{Chen2010}
that for appropriate $\lambda$, there are solutions of arbitrary periodicity.
In \cite{Wang2008} the authors rewrite (\ref{U2}) and (\ref{f2}) into an infinite system of nonlinear equations for the
Fourier coefficients of $f (\theta)$ and calculated numerically the first few
bifurcations. Chen et al \cite{Chen2010} observed that for even integers
$l = 2 n$, the the interaction potential
\begin{equation}
  U (f) (\theta) = \int_{S^1} \sin^l (\theta - \theta') f (\theta') d \theta',
\end{equation}
behave completely similar to the Maier-Saupe original potential and can be
reduced to a model in finite dimensional space, while for odd $l = 2 n - 1$
the obtained equation will be a nonlinear partial differential equation. By
reducing the Onsager equation to a system of ordinary differential equations,
they could prove the existence of auxiliary symmetric nematic solution for the
Onsager equation on $S^1$ and for all odd power potential kernel. More
recently, in \cite{Lucia2010} the authors studied the case $D = 2$
through cutting-off the Onsager kernel and reducing (\ref{U2}) and (\ref{f2}) to a finite dimensional system of nonlinear
equations, and obtain local bifurcation structure for this finite dimensional
approximation. In particular, they used a result of bifurcation by Crandall
and Rabinowitz \cite{Crandall71} for the general truncated trigonometric
kernel
\begin{equation}
  K (\theta, \theta') = - \sum_{n = 0}^N k_n \cos \nocomma 2 n (\theta -
  \theta') .
\end{equation}
The original Onsager kernel $| \sin (\theta - \theta') |$ on $S^1$ is
approximated by the above kernel for special
\begin{equation}
  k_n = \frac{1}{\pi}  \left( n^2 - \frac{1}{4} \right)^{- 1} .
\end{equation}
In this case the problem is reduced to finding the zeros of a finite
dimensional nonlinear problem.

In a new study of this problem in $D = 2$, X. Yu and the author {\cite{Nik15}}
obtained the following results for a general potential kernel (still
covering the original Onsager kernel):
\begin{itemizedot}
  \item The problem has a unique solution, which must be the constant
  solution, when $0 < \lambda < \lambda_0 \assign \frac{1}{|K|_{\infty} -
  k_0}$ where $k_0$ is the first mode of the Fourier expansion
  \begin{equation}
    K (\theta) = \sum_{k = 0}^{\infty} k_m \cos (2 \nocomma m \nocomma \theta).
  \end{equation}
  \item Two solutions bifurcate from the trivial solution at every $\lambda_m
  = - \frac{2}{k_m}$. The bifurcation is supercritical if $\frac{2 k_{2
  m}}{k_m} < 1$ and subcritical if $\frac{2 k_{2 m}}{k_m} > 1$. Furthermore,
  in the former case, the first pair of bifurcated solutions are stable and
  other bifurcated solutions are unstable, while in the latter case all
  bifurcated solutions are unstable.
  
  \item In particular, for the Onsager's model all bifurcations are
  supercritical. The first pair of bifurcated solutions are stable and other
  bifurcated solutions are unstable.
\end{itemizedot}

\section{Reformulation of the problem}

Let us formulate the problem in general dimension $D$ in terms of a nonlinear
map in a suitable space. The general Onsager model in $\mathbb{R}^D$ reads
\begin{equation}
  \label{U} U (f) (r) = \lambda \int_{S^{D - 1}} K (r, r') f (r') d \sigma
  (r'),
\end{equation}
\begin{equation}
  \label{f} f (r) = \left( \int_{S^{D - 1}} e^{- U (f)} d \sigma \right)^{- 1}
  e^{- U (f) (r)},
\end{equation}
where the potential kernel $K$ in (\ref{U}) is assumed to satisfy the
following properties:
\begin{equation}
  \label{K-prop} K (r, r') = K (- r, r') = K (r', r) = K (O (r), O (r')),
\end{equation}
for any rotation matrix $O ( \mathbb{R}^D)$. It is simply verified that this
conditions are satisfied by the original Onsager kernel (\ref{Onsager-K}).

We reformulate the system (\ref{U}) and (\ref{f}) into an abstract
equation involving a bounded, $(S)_+$ mapping, and then we generalize the classical Leray-Schauder degree to prove the existence and multiplicity of the
solution. Substitution (\ref{f}) in (\ref{U}) and canceling $f$ gives an equation for the potential $U(r)$:
\begin{equation}
  \label{Ur} U (r) = \lambda \left( \int_{S^{D - 1}} e^{- U} d \sigma
  \right)^{- 1}  \int_{S^{D - 1}} K (r, r') e^{- U (r')} d \sigma (r'),
\end{equation}
where $U$ enjoys the symmetric property $U (r) = U (- r)$. Note that once
(\ref{Ur}) is solved, $f (r)$ can be recovered from
\begin{equation}
  f (r) = \frac{e^{- U (r)}}{\int_{S^{D - 1}} e^{- U (r)} d \sigma} .
\end{equation}
Thus (\ref{Ur}) is equivalent to the original problem (\ref{U})-(\ref{f}).
Further reduction of the problem needs the following lemma.

\begin{lem}
  Under the symmetry assumptions $\left( \ref{K-prop} \right)$ on $K$, we have
  \[K (r, r') = F (|r - r' |),\] for some function $F$. In particular, this gives
  \begin{equation}
    \label{Kbar} \bar{K} = \frac{1}{|S^{D - 1} |}  \int_{S^{D - 1}} K (r, r')
    d \sigma (r') = \frac{1}{|S^{D - 1} |}  \int_{S^{D - 1}} K (r, r') d
    \sigma (r)
  \end{equation}
  is a constant.
\end{lem}

For a proof, we refer to \cite{Nik15}. Now we define
\begin{equation}
  \label{Khat} \hat{K} (r, r') = K (r, r') - \bar{K}
\end{equation}
with $\bar{K}$ is defined in (\ref{Kbar}). Now, for
\begin{equation}
  \label{Uhat} \hat{U} (r) = U (r) - \lambda \bar{K} .
\end{equation}
it is simply verified that (\ref{Ur}) is equivalent to the following system:
\begin{equation}
  V (r) = \frac{\lambda \int_{S^{D - 1}} \hat{K} (r, r') e^{- V (r')} d \sigma
  (r')}{\int_{S^{D - 1}} e^{- V} d \sigma}, \hspace{0.3cm} V (- r) = V (r),
  \hspace{0.2cm} \int_{S^{D - 1}} V (r) d \sigma = 0.
\end{equation}
Summarizing the above we reach:

\begin{lem}
  The original problem $\left( \ref{U} \right), \left( \ref{f} \right)$ is
  equivalent to the following problem.
  \begin{equation}
    \label{Vr} V (r) = \frac{\lambda \int_{S^{D - 1}} \hat{K} (r, r') e^{- V
    (r')} d \sigma (r')}{\int_{S^{D - 1}} e^{- V} d \sigma}, \hspace{1em} V (-
    r) = V (r), \hspace{1em} \int_{S^{D - 1}} V (r) d \sigma = 0.
  \end{equation}
\end{lem}

From now on, we will work with (\ref{Vr}) which is naturally a fixed point
problem. Let $G$ be the operator
\begin{equation}
  \label{G} G \nocomma (V) (r) = \beta (V)^{- 1}  \int_{S^{D - 1}} \hat{K} (r,
  r') e^{- V (r')} d \sigma (r'),
\end{equation}
where $\beta$ is
\begin{equation}
  \beta (V) = \int_{S^{D - 1}} e^{- V (r)} d \sigma (r),
\end{equation}
and $V$ belongs to $H_0 (S^{D - 1})$ where
\begin{equation}
  \left. \left. \label{H0} H_0 (S^{D - 1}) =\{u \in L^2 (S^{D - 1}), u (- r
  \right) = u (r), \int_{S^{D - 1}} u (r) \nocomma d \sigma = 0 \right\} .
\end{equation}
We employ the topological degree argument to study the structure of the
solutions of the equation
\begin{equation}
  \label{VlG} A (V) (r) \assign V (r) - \lambda G (V) (r),
\end{equation}
in terms of the parameter $\lambda$. We carry out the calculations in terms of
spherical harmonics on $S^{D - 1}$, which are the eigenfunctions of the
Laplace-Beltrami operator $- \Delta$ on $S^{D - 1}$. For $D = 2$, these
functions are just the usual trigonometric functions. Alternatively, for $r
\in S^{D - 1}$, the spherical harmonics $S_{nj} (D, r)$ can be defined by the
restriction of harmonic polynomials to the unit sphere $S^{D - 1}$. As it is
simply verified, see e.g. \cite{Calf95}, for $D \geq 3$ and given $n$, there
are exactly
\begin{equation}
  \label{N-Dn} N (D, n) = \frac{(2 n + D - 2)  (n + D - 3) !}{(D - 2) !n!}
\end{equation}
spherical harmonics $\{S_{nj} (D, r)\}_{j = 1}^{N (D, n)}$. An important class
of spherical harmonics on $S^{D - 1}$ consists of ones that are invariant
under the rotation of $S^{D - 2}$. These are usual Legendre polynomial $P_n
(D, t)$ for $t \in [- 1, 1]$. Alternatively, the Legendre polynomial $P_n$ can
be defined through the expansion of the potential $V (\rho) = (1 + \rho^2 - 2
\rho \cos \gamma)^{- \frac{D - 2}{2}}$ that is generated by a unit mass
located at distance $r = \frac{1}{\rho} > 1$ in terms of $\rho$ as
\begin{equation}
  V = \sum_{n = 0}^{\infty} P_n (D, \cos \gamma) \rho^n .
\end{equation}
\begin{prop}
  \label{Prop-K}The Onsager kernel in $\mathbb{R}^D$ defined by
  \begin{equation}
    \label{Ons-K} K (r, r') = | \sin \gamma |,
  \end{equation}
  where $\gamma$ is the angle between $r, r'$ has the expansion
  \begin{equation}
    \label{K-expan} K (\gamma) = - \sum_{n = 1}^{\infty} k_n P_{2 n} (D, \cos
    \gamma) + k_0,
  \end{equation}
  where $k_0 > 0$ and for $n \geq 1$, $k_n$ are positive and make a decreasing
  sequence, i.e.
  \begin{equation}
    \label{kn} k_n > 0, \hspace{0.5cm} k_n > k_{n + 1} .
  \end{equation}
\end{prop}

For the proof we need the following lemma.

\begin{lem}
  Let $C^{(\alpha)}_n (t)$ denote the Gegenbauer polynomials of order $n$ and
  $\alpha = \frac{D - 2}{2}$. We have
  \begin{equation}
    \label{C-n+2-n} \int_{- 1}^1 (1 - t^2)^{\alpha} C^{(\alpha)}_{n + 2} (t) =
    \frac{(n - 1)  (n + 2 \alpha)}{(n + 2)  (n + 2 \alpha + 3)}  \int_{- 1}^1
    (1 - t^2)^{\alpha} C^{(\alpha)}_n (t) .
  \end{equation}
\end{lem}

\begin{proof}
  Recall that $C^{(\alpha)}_n$ satisfies the Gegenbauer differential equation
  \begin{equation}
    \frac{d}{d \nocomma t}  \left[ (1 - t^2)^{\alpha + \frac{1}{2}} 
    \frac{d}{d \nocomma t} C^{(\alpha)}_n \right] + n (n + 2 \alpha)  (1 -
    x^2)^{\alpha - \frac{1}{2}} C^{(\alpha)}_n = 0.
  \end{equation}
  Multiply both sides of the Gegenbauer equation by $(1 - t^2)^{1 / 2}$ and
  integrate in $(- 1, 1)$ to obtain
  \begin{equation}
    \label{C-n} \int_{- 1}^1 (1 - t^2)^{\alpha} t \frac{d}{d \nocomma t}
    C^{(\alpha)}_n = - n (n + 2 \alpha)  \int_{- 1}^1 (1 - t^2)^{\alpha}
    C^{(\alpha)}_n .
  \end{equation}
  The same result holds for $n + 2$, that is,
  \begin{equation}
    \label{C-n+2} \int_{- 1}^1 (1 - t^2)^{\alpha} t \frac{d}{d \nocomma t}
    C^{(\alpha)}_{n + 2} = - (n + 2)  (n + 2 \alpha + 2)  \int_{- 1}^1 (1 -
    t^2)^{\alpha} C^{(\alpha)}_{n + 2} .
  \end{equation}
  Subtract formula (\ref{C-n}) from (\ref{C-n+2}) and use the identity
  \begin{equation}
    \frac{d}{d \nocomma t}  (C^{(\alpha)}_{n + 2} - C^{(\alpha)}_n) = 2 (n +
    \alpha + 1) C^{(\alpha)}_{n + 1},
  \end{equation}
  to reach
  \begin{eqnarray*}
    2 (n + \alpha + 1)  \int_{- 1}^1 (1 - t^2)^{\alpha} tC^{(\alpha)}_{n + 1}
    = n (n + 2 \alpha)  \int_{- 1}^1 (1 - t^2)^{\alpha} C^{(\alpha)}_n - &  &
    \\
    - (n + 2)  (n + 2 \alpha + 2)  \int_{- 1}^1 (1 - t^2)^{\alpha}
    C^{(\alpha)}_{n + 2} . &  & 
  \end{eqnarray*}
  Use the identity
  \begin{equation}
    2 (n + \alpha + 1) tC^{(\alpha)}_{n + 1} = (n + 2) C^{(\alpha)}_{n + 2} +
    (n + 2 \alpha) C^{\alpha}_n,
  \end{equation}
  to conclude
  \begin{eqnarray*}
    \int_{- 1}^1 (1 - t^2)^{\alpha}  [(n + 2) C^{(\alpha)}_{n + 2} + (n + 2
    \alpha) C^{\alpha}] = n (n + 2 \alpha)  \int_{- 1}^1 (1 - t^2)^{\alpha}
    C^{(\alpha)}_n - &  & \\
    - (n + 2)  (n + 2 \alpha + 2)  \int_{- 1}^1 (1 - t^2)^{\alpha}
    C^{(\alpha)}_{n + 2} . &  & 
  \end{eqnarray*}
  Now, a simple algebraic calculation gives (\ref{C-n+2-n}).
\end{proof}

Now let us return and prove the proposition (\ref{Prop-K}).

\begin{proof}
  {\dueto{of proposition (\ref{Prop-K})}}The set $P_n (D, \cos \gamma), n \geq
  0$ forms a complete system for the functions of $\gamma$ on $S^{D - 1}$ and
  thus the Onsager kernel has an expansion in terms of the functions in this
  set. On the other hand, since $( \ref{Ons-K})$ is even, the coefficients of
  the odd terms in the expansion are zero. Thus, the Onsager kernel has the
  expansion of the form (\ref{K-expan}). Since $K$ has a positive average,
  $k_0 > 0$. In order to show that $k_n$ enjoy (\ref{kn}), we use the explicit
  formula for $k_n$
  \begin{equation}
    \label{kn-integral} k_n = - \frac{\sigma_{D - 1} N (D, 2 n)}{\sigma_D} 
    \int_{- 1}^1 (1 - t^2)^{\frac{D - 2}{2}} P_{2 n} (D, t) d \nocomma t.
  \end{equation}
  We obtain
  \[ k_{n + 1} = - \frac{\sigma_{D - 1} N (D, 2 n + 2)}{\sigma_D}  \int_{-
     1}^1 (1 - t^2)^{\frac{D - 2}{2}} P_{2 n + 2} (D, t) d \nocomma t. \]
  Use the formula (\ref{C-n+2-n}) to obtain
  \begin{eqnarray*}
    k_{n + 1} = - \frac{\sigma_{D - 1} N (D, 2 n + 2)}{\sigma_D
    C^{(\alpha)}_{2 n + 2} (1)}  \int_{- 1}^1 (1 - t^2)^{\alpha}
    C^{(\alpha)}_{2 n + 2} (t) = &  & \\
    - \frac{\sigma_{D - 1} N (D, 2 n + 2)}{\sigma_D C^{(\alpha)}_{2 n + 2}
    (1)}  \frac{(2 n - 1)  (n + \alpha)}{(n + 1)  (2 n + 2 \alpha + 3)} 
    \int_{- 1}^1 (1 - t^2)^{\alpha} C^{(\alpha)}_{2 n} (t) = &  & \\
    - \frac{\sigma_{D - 1} N (D, 2 n + 2)}{\sigma_D C^{(\alpha)}_{2 n + 2}
    (1)}  \frac{(2 n - 1)  (n + \alpha) C^{(\alpha)}_{2 n} (1)}{(n + 1)  (2 n
    + 2 \alpha + 3)}  \int_{- 1}^1 (1 - t^2)^{\frac{D - 2}{2}} P_{2 n} (D, t)
    . &  & 
  \end{eqnarray*}
  Now use (\ref{kn-integral}) for the last integral above to write
  \begin{equation}
    \label{kn+1-C-kn} k_{n + 1} = \frac{C^{(\alpha)}_{2 n} (1) N (D, 2 n +
    2)}{C^{(\alpha)}_{2 n + 2} (1) N (D, 2 n)}  \frac{(2 n - 1)  (n +
    \alpha)}{(n + 1)  (2 n + 2 \alpha + 3)} k_n .
  \end{equation}
  Direct computation shows that $k_1 > 0$ and then all $k_n, n \geq 1$ are
  positive. Substitution $N (D, n)$ from (\ref{N-Dn}) and the following
  formula
  \begin{equation}
    \frac{C^{(\alpha)}_{2 n} (1)}{C^{(\alpha)}_{2 n + 2} (1)} = \frac{(2 n +
    1) \cdots (2 n + D - 3)}{(2 n + 3) \cdots (2 n + D - 1)},
  \end{equation}
  into (\ref{kn+1-C-kn}) we obtain
  \begin{equation}
    \label{kn+1-kn} k_{n + 1} = \frac{(2 n - 1)  (4 n + D + 2)  (2 n + D -
    2)}{2 (n + 1)  (4 n + D - 1)  (2 n + D + 1)} k_n .
  \end{equation}
  The coefficient of $k_n$ in (\ref{kn+1-kn}) is simply verified to be less
  than $1$, and therefore $(k_n), n \geq 1$ forms a positive decreasing
  sequence. Direct computation shows that $k_0 > 0$ and therefore $k_n$, $n
  \geq 1$ satisfy conditions (\ref{kn}).
\end{proof}

\section{Main Results}

In sequel, we assume $D \geq 3$, and that the potential kernel $K$ has the
expansion (\ref{K-expan}). As we observed above, this assumption covers the
original Onsager kernel. We systematically use the topological degree argument
and its bifurcation consequences for the equation
\begin{equation}
  \label{ulG} A (u) \assign u - g (u) = 0,
\end{equation}
where $g = \lambda G$ and $G$ is defined in (\ref{G}). Notice that the
classical Leray-Schauder degree fails to apply here because the map $A : H_0
(S^{D - 1}) \rightarrow H_0 (S^{D - 1})$ is not continuous in any neighborhood
of $0 \in H_0 (S^{D - 1})$. Let us show this by a simple example in $D = 3$.
For fixed $\bar{r} \in S^2$, let $u_n$ be the sequence
\begin{equation}
  u_n (r) = \left\{ \begin{array}{ll}
    \log (2 \pi (1 - \cos (1 / n))) & \cos^{- 1} (r. \bar{r})^{\nosymbol} \in
    \left( 0, \frac{1}{n} \right)\\
    0 & {\rm otherwise}
  \end{array} \right. .
\end{equation}
Obviously, $u_n \xrightarrow{H_0 (S^2)} 0$, while
\[ \lim_n G (u_n) = \lim_n  \frac{1}{2 \pi (1 - \cos (1 / n))}  \int_0^{1 / n}
   \int_0^{2 \pi} \hat{K} (\gamma) d \sigma \neq G (0) = 0. \]
\begin{thrm}
  \label{Hw}Assume that $K (\gamma)$ belongs to the class of Holder continuous
  maps, then the map $G : \Omega (\lambda) \subset H_0 \rightarrow H_0$ is
  continuous and compact where
  \begin{equation}
    \label{Omega} \Omega (\lambda) = \{u \in H_0 (S^{D - 1}), |u (r) | \leq
    \lambda || \hat{K} ||_{\infty} \} .
  \end{equation}
\end{thrm}

\begin{proof}
  Notice that the fixed point set of $g$ has the a priori bound
  \begin{equation}
    \label{u-bnd} |u (r) | \leq \lambda \| \hat{K} \|_{\infty} \beta (u)^{- 1}
    \int_{S^{D - 1}} e^{- u (r)} d \sigma = \lambda \| \hat{K} \|_{\infty} .
  \end{equation}
  The continuity of $g : \Omega \rightarrow H_0 (S^{D - 1})$ simply follows
  from the dominant convergence theorem. To prove that $g : \Omega (\lambda)
  \rightarrow H_0 (S^{D - 1})$ is compact, we show that it is the limit of a
  sequence of compact maps in $\Omega$. Let $\hat{K}_N$ be the truncated
  kernel
  \begin{equation}
    \hat{K}_N (\gamma) = - \sum_{n = 1}^N k_n P_{2 n} (D, \cos \nocomma
    \gamma) .
  \end{equation}
  and also let $g_N : \Omega (\lambda) \rightarrow H_0 (S^{D - 1})$ be the
  finite range operators as
  \begin{equation}
    g_N (u) (r) = \lambda \beta (u)^{- 1}  \int_{S^{D - 1}} \hat{K}_N (\gamma)
    e^{- u (r')} d \sigma (r') .
  \end{equation}
  It is seen that for $u \in \Omega$ we have
  \begin{eqnarray*}
    \|g (u) - g_N (u)\|^2 = \lambda^2 \beta (u)^{- 2}  \int_{S^{D - 1}} \left(
    \int_{S^{D - 1}} ( \hat{K} (\gamma) - \hat{K}_N (\gamma)) e^{- u (r')} d
    \sigma (r') \right)^2 d \sigma (r) \leq &  & \\
    \leq \max_{\gamma} | \hat{K} (\gamma) - \hat{K}_N (\gamma) | 2 \pi
    \lambda^2 \xrightarrow{N \rightarrow \infty} 0. &  & 
  \end{eqnarray*}
  Thus, $g$ is the uniform limit of a sequence of finite range operators $g_N$
  on the bounded set $\Omega \subset H_0$ and therefore a compact map on
  $\Omega$.
\end{proof}

\subsection{Degree Argument}

We give a generalization of the classical Leray-Schauder degree for the map $A
: \Omega \rightarrow H_0 (S^{D - 1})$ where $A$ is the map (\ref{ulG}) and
$\Omega$ is given in (\ref{Omega}). A generalization of the Browder degree for
$(S)_+$ mappings will appear soon \cite{Nik15b}.

Let $H$ be a separable Hilbert space with a fixed orthonormal basis
$\mathcal{H} = \{u^1, u^2, \cdots\}$ and the inner product $(,)$. The finite
dimensional subspaces $H_n \subset H$ for $n \geq 1$ are naturally defined by
$H_n \assign span \{u^1, \ldots, u^n \}$. For a given map $f : H \rightarrow
H$, the finite rank approximation $f_n : H \rightarrow H_n$ is defined by the
projection of $f (u)$ into $H_n$, that is,
\begin{equation}
  \label{finit-rank} f_n (u) = \sum_{k = 1}^n (f (u), u^k) u^k .
\end{equation}
It is simply verified that $(f (u), v) = (f_n (u), v)$ for arbitrary $v \in
H_n$. For our case, $H$ is the space $H_0 (S^{D - 1})$ defined in (\ref{H0})
and we chose $\mathcal{H}$ the set of spherical harmonics $\{S_{2 n \nocomma
j} (D, r)\}$. Notice that the following conditions are satisfied
\begin{itemizedot}
  \item For any $n \geq 1$, the set $\Omega_n \assign \Omega \cap H_n$ has a
  non-empty interior in $H_n$,
  
  \item The map $g : \Omega \rightarrow H$ is continuous and relatively
  compact,
  
  \item The solution set of the equation $A (u) = 0$ lies in $\Omega$ and
  furthermore $0 \nin A (\partial \Omega)$ where
  \begin{equation}
    \partial \Omega \assign \{u \in \Omega ; \|u\|_{\infty} = \lambda \|
    \hat{K} \|_{\infty} \} .
  \end{equation}
\end{itemizedot}
\begin{thrm}
  \label{stbl-prop}There exists $N_0 > 0$ such that
  \begin{equation}
    \label{deg-n-n+1} \deg (A_n, \Omega_n, 0) = \deg (A_{n + 1}, \Omega_{n +
    1}, 0), \hspace{0.5cm} \forall n \geq N_0 .
  \end{equation}
\end{thrm}

\begin{proof}
  Notice first that $\Omega_n$ has an open interior in $H_n$ for all $n$ and
  $A_n : \Omega_n \rightarrow H_n$ is continuous. Next we show that for
  sufficiently large $n$, there is no solution of the equation $A_n (u) = 0$
  for $u \in \partial \Omega_n$. Assuming contrary, there is a sequence
  $(z_n), z_n \in \partial \Omega$ such that $z_n = g_n (z_n)$. Since $g$ is
  completely continuous on $\Omega$, the sequence $g (z_n)$ converges (in a
  subsequence) to some $\zeta \in H$. Since $g_n (z_n) = \Pr_{H_n} g (z_n)$,
  it implies that $g_n (z_n) \xrightarrow{H} \zeta$ (in a subsequence) and
  then $z_n \xrightarrow{H} \zeta$. It is verified by the embedding of
  $L_{\infty}$ into $L^2 (S^{D - 1})$ that $\zeta \in \Omega$. Since $g$ is
  continuous in $\Omega$, we conclude $g (z_n) \xrightarrow{H} g (\zeta)$ and
  then $\zeta = g (\zeta)$. By a fact from the measure theory, we conclude $z_n
  \xrightarrow{{\rm pointwise}} \zeta$ almost everywhere. For arbitrary
  $\varepsilon > 0$, choose $n, r$ such that $|z_n (r) - \zeta (r) | <
  \frac{\varepsilon}{2}$ and $|z_n (r) | > \lambda \| \hat{K} \|_{\infty} -
  \frac{\varepsilon}{2}$. This implies that $| \zeta (r) | > \lambda \|
  \hat{K} \|_{\infty} - \varepsilon$ and then $\zeta \in \partial \Omega$, a
  contradiction. This establishes that for sufficiently large $n$, the
  solution of $A_n = 0$ does not occur on $\partial \Omega_n$. This allows us
  to define the classical Brouwer degree of the map $A_n$ restricted to
  $\Omega_n \cap H$ for sufficiently large $n$. Now define the map $B_{n + 1}$
  as follows:
  \begin{equation}
    \label{A-map} B_{n + 1} (u) = (A_n (u), (u, u^{n + 1}) u^{n + 1}) .
  \end{equation}
  Clearly by the classical properties of the Brouwer degree we have
  \begin{equation}
    \deg (B_{n + 1}, \Omega_{n + 1}, 0) = \deg (A_n, \Omega_{n + 1}, 0) .
  \end{equation}
  Now, consider the convex homotopy $h : [0, 1] \times \Omega_{n + 1}
  \rightarrow H_{n + 1}$
  \begin{equation}
    h (t) = (1 - t) A_{n + 1} + tB_{n + 1} .
  \end{equation}
  Note first that $0 \nin h (t) (z)$ for $z \in \partial \Omega_{n + 1}$ and
  $t = 0, 1$. If (\ref{deg-n-n+1}) does not hold then there exists a sequence
  $(t_n)$ for $t_n \in (0, 1)$ and $z_n \in \partial \Omega$ such that $h
  (t_n) (z_n) = 0$. This implies that
  \begin{equation}
    A_n (z_n) + t_n (g (z_n), u^n) u^n = 0.
  \end{equation}
  Since $\{z_n \} \subset \Omega$ is bounded and $g$ is compact on $\Omega$
  then $(g (z_n), u^n) \rightarrow 0$ and then $A_n (z_n) \rightarrow 0$ which
  implies in turn $z_n \rightarrow \zeta \in \partial \Omega$ and $A (z) = 0$,
  a contradiction. This completes the proof.
\end{proof}

By the aid of the theorem (\ref{stbl-prop}), we define the degree of $A$ on
$\Omega$ at $0$ by the limit in the following definition.

\begin{definition}
  \label{deg-Def}Under the above settings, the degree of $A$ in $\Omega$ at
  $0$ is defined as
  \begin{equation}
    \label{degree} \deg (A, \Omega, 0) = \lim_{n \rightarrow \infty} \deg
    (A_n, \Omega_n, 0) .
  \end{equation}
\end{definition}

In addition, we need to define the class of admissible homotopy for the
generalized degree (\ref{degree}).

\begin{definition}
  The map $h : [0, 1] \times \Omega \rightarrow H$ defined by the relation $h
  (t) (u) = u - G (t) (u)$ is called an admissible homotopy if $h$ satisfies
  the following conditions
  \begin{itemizedot}
    \item $h$ in continuous in $[0, 1] \times \Omega$,
    
    \item The solution set of the equation $h (t) (u) = 0$ lies in $\Omega$
    and furthermore $0 \nin h (t)  (\partial \Omega)$ for all $t \in [0, 1]$,
    
    \item the map $g$ is compact, i.e., $g ([0, 1]) : \Omega \to H$ is
    compact, and in addition for every $\Omega' \subset \Omega$ and
    $\varepsilon > 0$ there exist $\delta = \delta (\varepsilon, \Omega')$
    such that
    \begin{equation}
      |t - s| < \delta \Rightarrow \|g (t) (x) - g (s) (x)\| < \varepsilon,
      \hspace{0.4cm} x \in \Omega' .
    \end{equation}
  \end{itemizedot}
\end{definition}

The above definition of admissible homotopy is very similar to one defined for $(S)_+$ mappings by F. Browder \cite{Browder82}. We
verifies that the definition (\ref{deg-Def}) satisfies the classical properties of
a topological degree. In particular, we are interested in the homotopy
invariance and the solvability properties.

\begin{thrm}
  Under the above setting, the degree $h_t : \Omega \rightarrow H$ is constant
  with respect to $t \in [0, 1]$. Furthermore, if $\deg (h (t), \Omega, 0)
  \neq 0$ for some $t \in [0, 1]$ then there exist $u = u (t) \in \Omega$ such
  that $h (t) (u (t)) = 0$.
\end{thrm}

\begin{proof}
  We show first that for some $N_0 > 0$ and for $n \geq N_0$, the finite
  rank approximation $h_n$ of the homotopy $h$ satisfies the condition $0 \nin
  h_n ([0, 1])  (\partial \Omega_n)$. Assuming contrary, there exists a
  sequence $t_n$ and $z_n \in \partial \Omega_n$ such that $h_n (t_n) (z_n) =
  0$. Since $t_n$ converges (in a subsequence) to some $\bar{t} \in [0, 1]$
  and since $g ( \bar{t}) : \Omega \rightarrow H$ is compact, we conclude $\{g
  ( \bar{t}) (z_n)\}$ converges (in a subsequence) to some $\zeta \in H$.
  Hence, we can write
  \begin{eqnarray*}
    \|z_n - \zeta \| = \|g_n (t_n) (z_n) - g ( \bar{t}) (z_n) + g ( \bar{t})
    (z_n) - \zeta \| \leq &  & \\
    \leq \|g_n (t_n) (z_n) - g ( \bar{t}) (z_n)\| + \|g ( \bar{t}) (z_n) -
    \zeta \| = &  & \\
    = \|g_n (t_n) (z_n) - g ( \bar{t}) (z_n)\| + o (1) . &  & 
  \end{eqnarray*}
  On the other hand, since $g$ is a compact transformation, we have 
  
  \begin{eqnarray*}
    \|g_n (t_n) (z_n) - g ( \bar{t}) (z_n)\| = \|g_n (t_n) (z_n) - g (t_n)
    (z_n) + g (t_n) (z_n) - g ( \bar{t}) (z_n)\| \leq &  & \\
    \|g_n (t_n) (z_n) - g (t_n) (z_n)\| + \|g (t_n) (z_n) - g ( \bar{t})
    (z_n)\| \leq o (1) . &  & 
  \end{eqnarray*}
  and thus $\|g (t_n) (z_n) - g ( \bar{t}) (z_n)\| \rightarrow 0$.
  Recall that $g_n = \Pr_{H_n} g$, and then by the relation $\|g_n (t_n) (z_n) - g (t_n)
  (z_n)\| \rightarrow 0$ we conclude $z_n \xrightarrow{H} \zeta$ (in a
  subsequence). A measure theoretic argument that we have used in the proof of
  the Theorem (\ref{stbl-prop}) implies $\zeta \in \partial \Omega$. We
  conclude finally $0 = h ( \bar{t}) (\zeta)$ for $\zeta \in \partial \Omega$
  which contradicts the second assumption on $h$. Furthermore, one can employ an argument
  similar to one presented in the proof of the Theorem(\ref{stbl-prop}) to shows that the degree $\deg (h_n
  (t), \Omega_n, 0)$ is stable with respect to $n$ for any $t \in [0, 1]$.
  This allows us to write 
  \begin{equation}
    \deg (h (t), \Omega, 0) = \deg (h_n (t), \Omega_n, 0),
  \end{equation}
  for sufficiently large $n$. Now assume that for $t_1, t_2 \in [0, 1]$, we have
  \begin{equation}
    \deg (h (t_1), \Omega, 0) \neq \deg (h (t_2), \Omega, 0) .
  \end{equation}
  Thus we can choose $n$ so large that 
  \begin{equation}
    \deg (h_n (t_1), \Omega_n, 0) \neq \deg (h_n (t_2), \Omega_n, 0) .
  \end{equation}
  On the other hand, since $0 \nin h_n ([0, 1])  (\partial \Omega)$, the
  homotopy invariance property of the Brouwer degree implies
  \begin{equation}
    \deg (h_n (t_1), \Omega_n, 0) = \deg (h_n (t_2), \Omega_n, 0) .
  \end{equation}
  which is a contradiction.
  
  Now we prove the second part of the theorem. Assume $\text{deg}(h(t),\Omega,0)=0$ for a fixed $t\in [0,1]$.
 According to the definition (\ref{degree}) and the solvability property of the Brouwer degree, there exists a sequence $(u_n (t)) \subset \Omega_n$ such that $h_n (t)  (u_n (t)) = 0$. Since $g$ is compact, the sequence $(g (t, u_n (t)))$
  converges (in a subsequence) to some $u (t) \in H$. This implies that $g_n
  (t, u_n (t))$ converges to $u (t)$ and therefore $u_n (t) \xrightarrow{H} u
  (t) \in \Omega$. Since $h$ is continuous, we have $h (t) (u (t)) = 0$ and
  this completes the proof.
\end{proof}

\subsection{Phase Transition}

We prove that thres is $\lambda_0 > 0$ such that the isotropic phase is the unique solution to the system
(\ref{U}) and (\ref{f}) for $\lambda < \lambda_0$. In addition we derive a sequence of critical values $\lambda_n$ for which 
the axisymmetric nematic phases will bifurcate from the trivial solution. First the following lemma.

\begin{lem}
  \label{iso-lem}Let $L : H_0 (S^{D - 1}) \rightarrow H_0 (S^{D - 1})$ be the
  map
  \begin{equation}
    \label{L} L (u) (r) = \frac{- 1}{\sigma_D}  \int_{S^{D - 1}} \hat{K}
    (\gamma) u (r') d \nocomma \sigma (r') .
  \end{equation}
  If $\lambda$ is not a characteristic value of $L$ then $\bar{u} = 0$ is an
  isolated solution of $A(u)=0$.
\end{lem}

\begin{proof}
 It is simply verified by the dominant convergence theorem that if  $u \in 2 \Omega(\lambda)$ then
    \begin{equation}
    \|G (u) - L (u)\|_{L^2 (S^{D - 1})} = o (\|u\|_{L^2 (S^{D - 1})}) .
  \end{equation}
Fix $\lambda$ a non-characteristic value of $L$. If $\bar{u} = 0$ is not isolated then choose a sequence $(\lambda_n, u_n)$
  such that $\lambda_n \rightarrow \lambda$, $u_n \rightarrow 0$ and $u_n = g (u_n)$. On the other hand, we have
  \begin{equation}
    0 = \|u_n - g (u_n)\| \geq \|u_n - \lambda L (u_n)\| - | \lambda -
    \lambda_n | \|L\| \|u\| -\|g (u_n) - \lambda_n \nocomma L (u_n) \|.
  \end{equation}
  Since $\lambda$ is not a characteristic value of $L$, there exist $k > 0$
  such that \[\|u_n - \lambda L (u_n)\| > k \|u_n \|.\] Take $| \lambda_n -
  \lambda |$ very small and then
  \begin{equation}
    k \|u_n \| \noplus + o (\|u_n \|) < 0,
  \end{equation}
  that is a contradiction.
\end{proof}

\begin{thrm}
  Under the above setting, there exist $\lambda_0 > 0$, such that the equation $A(u)=0$ has the unique solution $\bar{u} = 0$ in the class of axially symmetric solutions for $0 < \lambda <  \lambda_0$.
\end{thrm}

\begin{proof}
  Let $\sigma_D$ stands for the surface of the unit sphere in $\mathbb{R}^D$. For $R = \lambda \sqrt{\sigma_D} \| \hat{K} \|_{\infty}$ the equation
  \begin{equation}
    u - tg (u) = 0,
  \end{equation}
  has no solution on the sphere $S_R$ for $t \in [0, 1]$. In fact we have
  $|| u ||_{\infty} \leq R / \sqrt{\sigma_D}$ and $\|u\|_{L^2} \leq$R. By the
  homotopy invariance property of degree we conclude:
  \begin{equation}
    \label{deg} \deg ({\rm Id} - g, B_R, 0) = \deg ({\rm Id} - tg, B_R, 0) =
    \deg ({\rm Id}, B_R, 0) = + 1
  \end{equation}
  We show that the index of the trivial solution $\bar{u}$ is $+ 1$. We use the expansion of functions in $H_0 (S^{D - 1})$ 
  in terms of the orthonormal spherical harmonics $\{S_{2 n \nocomma j} (D, r)\}$ for $j = 1,
  \ldots, N (D, n)$. Let us write for $u \in H_0 (S^{D - 1})$ the expansion 
  \begin{equation}
    \label{u-S} u (r) = \sum_{n = 1}^{\infty} \sum_{j = 1}^{N (D, 2 n)} u_{n
    \nocomma j} S_{2 n \nocomma j} (D, r) ,
  \end{equation}
  for some  coefficients $u_{n  j}$.  With regards to (\ref{u-S}), it is more convenient in our calculations to write $u = u (u_{n \nocomma
  j})$. Calculation of the entries of the Jacobian matrix of $g$ at $u =
  \bar{u}$ gives
  \begin{eqnarray*}
    \frac{\partial}{\partial u_{n \nocomma j}} g ( \bar{u}) = -
    \frac{\lambda}{\sigma_D}  \int_{S^{D - 1}} \hat{K} (\gamma) S_{2 n
    \nocomma j} (D, r') d \nocomma \sigma (r') = &  & \\
    \frac{\lambda k_n}{\sigma_D}  \int_{S^{D - 1}} P_{2 n} (D, \cos \nocomma
    \gamma) S_{2 n \nocomma j} (D, r') d \sigma (r') = &  & \\
    = \frac{\lambda k_n}{N (D, 2 n)} S_{2 n \nocomma j} (D, r) . &  & 
  \end{eqnarray*}
  This implies that the infinite dimensional Jacobian matrix of $g ( \bar{u})$ in the
  bases of $\{S_{n  j} \}$ with $n$ an even number has the form
  \begin{equation}
    J_G = {\rm diag} \left( \frac{\lambda k_n}{N (D, 2 n)} \right) .
  \end{equation}
  Therefore, if $\tilde{\lambda}_0 \leq Dk_1^{- 1}$, then $\text{ind} ( \bar{u}, \lambda) = 1 $  if $0 < \lambda <  \tilde{\lambda}_0$. The axially symmetric
solutions are functions which are symmetric with respect to the rotation of
$S^{D - 2}$ around any point of $S^{D - 1}$. For the fixed $r \in S^{D - 1}$,
if $\theta$ denotes the angel between arbitrary point $r' \in S^{D - 1}$ and
$r$, then $u$ can be expanded in terms of Legendre polynomials $P_{2 n}$ as
\begin{equation}
  \label{axisym-u} u (\theta) = \sum_{n = 1}^{\infty} u_n P_{2 n} (D, \cos
  \theta) .
\end{equation}
In this case $g$ has the simpler form:
\begin{equation}
  g (u) (\theta) = \lambda \int_0^{\pi} \hat{K} (\gamma)  \tilde{g} (\theta')
  d \nocomma \theta',
\end{equation}
where $\tilde{g}$ is defined as
\begin{equation}
  \tilde{g} (\theta) = \frac{e^{- u (\theta)} \sin^{D - 2}
  (\theta)}{\int_0^{\pi} e^{- u (\theta)} \sin^{D - 2} (\theta) d \nocomma
  \theta} .
\end{equation}
The calculations reduces to what we have carried out for the Onsager problem
in case $D = 2$, see \cite{Nik15}. The Jacobian matrix entries for the
trivial solution is obtained as:
\begin{equation}
  \label{ubar-der} \frac{\partial}{\partial u_n} g ( \bar{u}) = - \lambda
  \frac{\sigma_{D - 1}}{\sigma_D}  \int_0^{\pi} \hat{K} (\gamma) P_{2 n} (D,
  \cos \nocomma \theta') \sin^{D - 2} \nocomma (\theta') d \nocomma \theta' =
  \frac{\lambda k_n}{N (D, 2 n)} P_{2 n} (D, \cos \nocomma \theta) .
\end{equation}
The calculation for a non-trivial solutions also is carried out as
\begin{eqnarray*}
  \left\langle \frac{\partial}{\partial u_{n \nocomma}} g (u), P_{2 m} (D,
  \cos \theta) \right\rangle = \lambda k_m  \left\{ \int_0^{\pi} \tilde{g}
  (\theta) P_{2 n} (D, \cos \nocomma \theta) P_{2 m} (D, \cos \nocomma \theta)
  d \nocomma \theta - \right. &  & \\
  - \left. \int_0^{\pi} \tilde{g} (\theta) P_{2 n} (D, \cos \theta) d \nocomma
  \theta \int_0^{\pi} \tilde{g} (\theta) P_{2 m} (D, \cos \nocomma \theta) d
  \nocomma \theta \right\} . &  & 
\end{eqnarray*}
The bracket in the right hand side of the above identity can be calculated
using a Gruss type inequality \cite{Dra00} which states that for
functions $a, b \in L^{\infty} (D)$ and the probability measure $\mu$, we have
the inequality
\begin{equation}
  \left| \int_D a (x) b (x) d \nocomma \mu - \left( \int_D a (x) d \nocomma
  \mu \right) \left( \int_D b (x) d \nocomma \mu \right) \right| \leq
  \|a\|_{L^{\infty}} \|b\|_{L^{\infty}} .
\end{equation}
According to the above inequality we obtain
\[ \left| \int_0^{\pi} \tilde{g} (\theta) P_{2 n} (D, \cos \nocomma \theta)
   P_{2 m} (D, \cos \nocomma \theta) d \nocomma \theta - \int_0^{\pi}
   \tilde{g} (\theta) P_{2 n} (D, \cos \theta) d \nocomma \theta \int_0^{\pi}
   \tilde{g} (\theta) P_{2 m} (D, \cos \nocomma \theta) d \nocomma \theta
   \right| \leq 1, \]
and then finally we reach
\begin{equation}
  \label{u-der} \left| \left\langle \frac{\partial}{\partial u_{n \nocomma}} g
  (u), P_{2 m} (D, \cos \theta) \right\rangle \right| \leq \lambda k_m .
\end{equation}
The above calculation establishes the fact ${\rm ind} (u, \lambda) = + 1$ for
any axially symmetric solution $u$ of $A(u)=0$ and for $0 < \lambda <\lambda_0$ where
\begin{equation}
  \label{l0} \lambda_0 = \left( \sum_{m = 1}^{\infty} k_m \right)^{-
  1} .
\end{equation}
The bound (\ref{l0}) is obtained based on the assumption $k_n > 0$ in
(\ref{kn}). Since the degree of $A$ is $+ 1$ according to (\ref{deg}), we conclude that
  the trivial solution $\bar{u}$ is the unique solution in the class of axially symmetric solutions for $0 < \lambda <
  \lambda_0$.
\end{proof}
It can be shown that the above theorem holds in general case (not necessarily for axially symmetric solutions). The calculation in this case gives a bound $\tilde{\lambda}_0 < \lambda_0$. We have the following theorem.
\begin{thrm}\label{appthm}
  The equation $A(u)=0$ has no non-trivial solution for $0 < \lambda < \tilde{\lambda}_0$ where $\tilde{\lambda}_0=\frac{1}{5} \| \hat{K} \|_{\infty}^{- 1}$.
\end{thrm}

The proof is based on the fact that the index of every solution of $(A(u)=0$ is $+1$ for $0<\lambda <\tilde{\lambda}_0$. The calculation is given in the appendix. 

To prove the existence of nematic phases for the Onsager model
(\ref{U}) and (\ref{f}), we use the following lemma and the degree argument we
established in the previous section.

\begin{lem}
  \label{bif-lem} Let $\sigma (L)$ be
  the spectrum of the map $L$ defined in $\left( \ref{L} \right)$ and let $\bar{\lambda} \in \sigma (L)$. If for the
  pair $(\lambda, \mu)$ where $0 < \lambda < \bar{\lambda} < \mu$ we have
  \begin{equation}
    {\rm ind} ({\rm Id} - \lambda L, 0) {\rm ind} ({\rm Id} - \mu
    L, 0) < 0,
  \end{equation}
  then $\bar{\lambda}$ is a bifurcation point for $\left( \ref{ulG} \right)$.
\end{lem}

\begin{proof}
  Assuming contrary, the value $( \bar{\lambda}, 0)$ is isolated due to the lemma (\ref{iso-lem}) and
  then there exists $\varepsilon > 0$ such that for $\lambda \in (
  \bar{\lambda} - \varepsilon, \bar{\lambda} + \varepsilon)$, the trivial
  solution $\bar{u}$ is the unique solution of $A (u) = 0$ for $u \in
  B_{\varepsilon} ( \bar{u})$ . This implies that there is no solution laying
  on $\partial B_{\varepsilon / 2}$ for $\lambda \in ( \bar{\lambda} -
  \varepsilon, \bar{\lambda} + \varepsilon)$ and then by the homotopy
  invariance property of the degree (\ref{degree}), the index ${\rm ind}
  ({\rm Id} - \lambda L, 0)$ is constant that is a contradiction.
\end{proof}

By the Lemma (\ref{bif-lem}) and calculations
(\ref{ubar-der}) and (\ref{u-der}), we are able to prove the existence of
infinitely many axially symmetric nematic phases of the Onsager model. All
nematic phases are bifurcation solutions of (\ref{ulG}) from the trivial
solution $\bar{u} = 0$. The argument is completely similar to one we employed
for the model in $D = 2$, see \cite{Nik15}. In particular we have the
following theorem.

\begin{thrm}
  There exists a sequence of axially symmetric solution of $\left( \ref{ulG}
  \right)$ in $H_0 (S^{D - 1})$ bifurcating from the trivial solution 
  $\bar{u}$ at the critical values \[\lambda_n = N (D, 2 n) k_n^{- 1}.\] The
  multiplicity of the bifurcating solutions at each bifurcating point
  $\lambda_n$ is exactly equal 2 and the first bifurcation solution is stable.
\end{thrm}

\begin{proof}
  Notice that $\lambda_n$ are the eigenvalues of the operator $L$ defined in
  (\ref{L}). According to the calculation (\ref{ubar-der}), the ${\rm ind} (
  \bar{u}, \lambda)$ changes the sign when $\lambda$ passes through
  $\lambda_n$, i.e., for sufficiently small $\varepsilon > 0$ and $\lambda \in
  (\lambda_n - \varepsilon, \lambda_n)$ and $\mu \in (\lambda_n, \lambda_n +
  \varepsilon)$ we have
  \begin{equation}
    {\rm ind} ( \bar{u}, \lambda) {\rm ind} ( \bar{u}, \mu) = - 1.
  \end{equation}
  Therefore, due to the lemma (\ref{bif-lem}), the values $\lambda_n$ are bifurcation points.
  Since $L$ is a self-adjoint operator, the algebraic and geometric
  multiplicity of eigenvalues of $L$ coincide. It is simply seen that the
  unique eigenfunction (up to the normalization) of $L$ at $\lambda_n$ is
  $P_{2 n} (D, \cos \nocomma \theta)$ and then $\lambda_n$ is a simple
  eigenvalue for $L$. Due to the Theorem 4.2 in \cite{Sattinger1973}, we
  conclude that there exist exactly two solutions bifurcating from the trivial solution $\bar{u}$ at all critical values 
  $\lambda_n$ and furthermore the first bifurcation solution at $\lambda_1$ is
  stable; see also \cite{Nik15}.
\end{proof}
\appendix
\section{Proof of the theorem \ref{appthm}}
Let $u$ be an arbitrary solution to $A(u)=0$, we have
  \begin{eqnarray*}
    \left\langle \frac{\partial}{\partial u_{n \nocomma j}} g (u), S_{2 m
    \nocomma l} (D, r) \right\rangle = - \lambda \int_{S^{D - 1}} \int_{S^{D -
    1}} \hat{K} (\gamma) f (r') S_{2 n \nocomma j} (D, r') \overline{S_{2 m
    \nocomma l} (D, \nocomma r)} d \nocomma \sigma (r') + &  & \\
    + \lambda \int_{S^{D - 1}} f (r) S_{2 n \nocomma j} (D, r)  \int_{S^{D -
    1}} \int_{S^{D - 1}} \hat{K} (\gamma) f (r') \overline{S_{2 m \nocomma l}
    (D, r)} d \sigma (r') = &  & \\
    = \frac{\sigma_D \lambda k_m}{N (D, 2 m)}  \int_{S^{D - 1}} f (r) S_{2 n
    \nocomma j} (D, r) \overline{S_{2 m \nocomma l} (D, r)} - &  & \\
    - \frac{\sigma_D \lambda k_m}{N (D, 2 m)}  \int_{S^{D - 1}} f (r) S_{2 n
    \nocomma j} (D, r)  \int_{S^{D - 1}} f (r) \overline{S_{2 m \nocomma l}
    (D, r)} . &  & 
  \end{eqnarray*}
  Let $b_{n  j}^{m  l}$ denote the following expression:
  \[ b_{n \nocomma j}^{m \nocomma l} = \int_{S^{D - 1}} f (r) S_{2 n \nocomma
     j} (D, r) \overline{S_{2 m \nocomma l}} (D, r) - \int_{S^{D - 1}} f (r)
     S_{2 n \nocomma j} (D, r)  \int_{S^{D - 1}} f (r) \overline{S_{2 m
     \nocomma l} (D, r)} . \]
  An estimate for $b_{n \nocomma j}^{m \nocomma l}$ (not necessarily optimal)
  using the a priori estimate (\ref{u-bnd}) is as follows:
  \begin{equation}
    |b_{n \nocomma j}^{m \nocomma l} | \leq \frac{2 e^{4 \lambda || K
    ||_{\infty}}}{\sigma_D} .
  \end{equation}
  This implies in turn
  \[ \frac{\partial}{\partial u_{n \nocomma j}} g (u) (r) = \sum_{m =
     1}^{\infty} \sum_{l = 1}^{N (D, 2 m)} a_{m \nocomma l} S_{2 m \nocomma l}
     (D, r), \]
  where $a_{m \nocomma l}$ has the following bound:
  \begin{equation}
    |a_{m \nocomma l} | \leq \frac{2 \lambda k_m e^{4 \lambda || K
    ||_{\infty}}}{N (D, 2 m)} .
  \end{equation}
  The following estimate gives a bound for which ${\rm ind} (u, \lambda) = +
  1$ for any possible solution of (\ref{ulG}):
  \begin{equation}
    \label{l-bnd} \lambda e^{4 \nocomma \lambda || K ||_{\infty}}  \sum_{m =
    0}^{\infty} \sum_{l = 1}^{N (D, 2 m)} \frac{k_m}{N (D, 2 m)} = \lambda
    e^{4 \nocomma \lambda || K ||_{\infty}}  \sum_{m = 1}^{\infty} k_m <
    \frac{1}{2} .
  \end{equation}
  Using the estimate (\ref{l-bnd}) we conclude that the index of every
  possible solution of (\ref{ulG}) is $+ 1$ for $0 < \lambda < \lambda_0$
  where
  \begin{equation}
    \label{l0-1} \lambda_0 = \frac{1}{5} \| \hat{K} \|_{\infty}^{- 1} .
  \end{equation}

\end{document}